\documentclass[letterpaper, 10 pt, journal, twoside]{IEEEtran}
\usepackage{url}
\usepackage{graphicx}
\usepackage{wrapfig}
\usepackage[format=plain,font=footnotesize,labelfont=bf,labelsep=period]{caption}
\usepackage{sidecap} 
\usepackage{subfig}
\usepackage[export]{adjustbox}
\usepackage[font=small]{caption}
\usepackage{float}

\usepackage{amsmath} 
\usepackage{amssymb}  
\usepackage{amsthm}
\usepackage{mathtools}
\usepackage{pdfsync}
\usepackage[normalem]{ulem}
\usepackage{paralist}	
\usepackage[space]{grffile} 
\usepackage{color}

\newtheorem{theorem}{Theorem}

\newtheorem{lemma}{Lemma}
\theoremstyle{definition}
\newtheorem{definition}{Definition}
\theoremstyle{remark}
\newtheorem{remark}{Remark}
\theoremstyle{definition}

\theoremstyle{definition}

\newcommand{\newsec}[1]{
\vspace{0.2cm} 
\noindent \textbf{#1}
}

\newcommand{\R}{\mathbb{R}}
\renewcommand{\S}{\mathcal{S}}

\newcommand{\he}{h_e}
\newcommand{\Se}{\S_e}
\newcommand{\Sx}{\S_x}
\newcommand{\Su}{\S_u}
\newcommand{\pe}{p_e}
\newcommand{\de}{d_e}

\IEEEoverridecommandlockouts
\allowdisplaybreaks

\begin{document}

\theoremstyle{definition}
\theoremstyle{remark}
\theoremstyle{definition}
\theoremstyle{definition}

\title{\LARGE \bf Integral Control Barrier Functions\\for Dynamically Defined Control Laws}
\author{Aaron D. Ames$^1$, Gennaro Notomista$^2$, Yorai Wardi$^3$, and Magnus Egerstedt$^3$
\thanks{\textcopyright 2020 IEEE. Personal use of this material is permitted. Permission from IEEE must be obtained for all other uses, in any current or future media, including reprinting/republishing this material for advertising or promotional purposes, creating new collective works, for resale or redistribution to servers or lists, or reuse of any copyrighted component of this work in other works.}
\thanks{$^1$Mechanical and Civil Engineering, Control and Dynamical Systems, Caltech, Pasadena, CA 91125.
        {\tt\small ames@caltech.edu}}
\thanks{$^2$George W. Woodruff School of Mechanical Engineering, Georgia Tech, Atlanta, GA 30332. {\tt\small g.notomista@gatech.edu}}
\thanks{$^3$School of Electrical and Computer Engineering, Georgia Tech, Atlanta, GA 30332. {\tt\small \{ywardi,magnus\}@ece.gatech.edu}} 
}
\maketitle
\thispagestyle{empty}

\begin{abstract}
This paper introduces integral control barrier functions (I-CBFs) as a means to enable the safety-critical integral control of nonlinear systems. Importantly, I-CBFs allow for the holistic encoding of both state constraints and input bounds in a single framework.  We demonstrate this by applying them to a dynamically defined tracking controller, thereby enforcing safety in state and input through a minimally invasive I-CBF controller framed as a quadratic program.
\end{abstract}


\section{Introduction}

Control Barrier Functions (CBFs) have proven to be effective at enforcing safety in nonlinear systems. The goal of this paper is to introduce a new form of CBFs that are suitable for feedback systems where the controller is defined  by a differential equation \cite{Isidori90}, i.e., integral control \cite{freeman1995robust,khalil2000universal,jiang2001robust}. 
The motivation comes from a dynamically defined tracking controller that has displayed effective tracking convergence \cite{Wardi19}, but also can exhibit large overshoots in the input controls at early transient phases. We investigate a CBF-based approach to limit these overshoots, i.e., satisfy input bounds, while also satisfying safety constraints on the state.  The results are presented in a more-general setting of pointwise constraints on input and state encoded by a novel form of CBF.

The general CBF framework (see \cite{Ames19} and references therein) requires solving an optimization problem at each time to enforce the CBF condition. If the dynamics of the state are control affine, this optimization problem is a convex quadratic program (QP), solvable in real time. Yet this real-time solvability hinges on the control affine nature of the dynamics. Moreover, when input constraints are also present, adding them to the QP can result in its infeasibility.  This paper will address both of these existing limitations in the context of integral control.

This paper introduces a new class of CBFs: \emph{integral control barrier functions (I-CBFs)}.  These barrier functions are developed in the context of systems with dynamically defined controllers, i.e., systems where the evolution of the state and input are described by an ordinary differential equation.  As a result, I-CBFs are defined on both the state and input, allowing for the inclusion of the input in the safety conditions encoded by this function. A related formulation is presented in \cite{huang2019guaranteed}, where control-dependent CBFs are defined: these, unlike I-CBFs, are considered for inputs with bounded time-derivative and, importantly, the integration with nominal tracking controllers was not considered. In this paper, given CBFs and I-CBFs, we present a resulting controller that guarantees safety in state and input, while minimally modifying a nominal dynamically defined controller. Additionally, we explicitly explore systems with both state and input constraints, and demonstrate how these can be unified through the I-CBF framework.

\section{Background Material}
\label{sec:background}

This section summarizes established results regarding safety-critical control via control barrier functions (CBFs), and the tracking-control technique that will inspire the main result of this paper. 
Herein, we consider a nonlinear control system defined by the differential equation:
\begin{equation}
\label{eqn:controlsys}
    \dot{x}(t)=f(x(t),u(t)),
\end{equation}
where $t\geq 0$, the state is $x(t)\in \R^n$, the input is  $u(t)\in \R^m$, and an initial state $x_{0}:=x(0)\in \R^n$ is given.
Assume that the function $f:\R^n\times \R^m\rightarrow \R^n$ is continuously differentiable,  
and it
satisfies sufficient conditions for the existence of a unique solution to \eqref{eqn:controlsys} over all $t\geq 0$ for every bounded, piecewise-continuous control $u(t)$ and $x_{0}\in \R^n$.

\subsection{Control Barrier Functions}
\label{ssec:cbfs}

Control Barrier Functions (CBFs) ensure the forward invariance of a set $\S\subset\R^n$, in which case the system is considered \emph{safe}. The CBF framework was introduced in \cite{Ames14,Ames17} where it was applied to adaptive cruise control, and has since been applied to a variety of application domains, including: automotive safety \cite{jankovic2018robust,xu2015robustness}, robotics \cite{nguyen2016exponential,wang2016multi} and multi-robot systems \cite{Wang17,lindemann2018control}.
See \cite{Ames19} for a recent survey. 
Note that alternative approaches to CBFs include reference governors \cite{gilbert2002nonlinear,nicotra2018explicit} and integral barrier Lyapunov functions \cite{tee2012control}; yet, the former is computationally more expensive, whereas the latter only applies to systems in strict-feedback form and does not explicitly consider input bounds.

Let $h:\R^n\rightarrow \R$ be a continuously-differentiable function 
such that $0$ is a regular value.  We will consider the set, $\S\subset\R^n$, given as the 0-superlevel set of $h$, i.e., $\S$ is defined by:
\begin{eqnarray}
\S & := & \{ x \in \R^n ~ : ~ h(x) \geq 0 \} , \nonumber\\
\partial \S & = & \{ x \in \R^n ~ : ~ h(x) = 0 \}, \nonumber\\
\mathrm{Int}(\S) & = & \{ x \in \R^n ~ : ~ h(x) > 0 \}. \nonumber
\end{eqnarray}
Consider a feedback control law applied to \eqref{eqn:controlsys} of the form:
\begin{equation}
\label{eqn:feedbackcont}
    u(t)=k(x(t)),
\end{equation}
with $k:\R^n\rightarrow \R^m$ continuous.  Under this controller, we say that $x(t)$ is {\it safe} if and only if
$x_0 \in \S$ implies that $x(t)\in \S$ for all $t \geq 0$, i.e., $\S$ is \emph{forward invariant}, i.e., $\S$ is safe.  

The function $h$ is a {\it control barrier function (CBF)} if  there exists a continuous extended class-${\cal K}$ function $\gamma:\R\rightarrow \R$ (monotone increasing, $\gamma(0)=0$) together with a feedback controller \eqref{eqn:feedbackcont} such that along every trajectory of the closed-loop system,  the following CBF condition holds \cite{Ames17}: 
\begin{equation}
\label{eqn:CBFcond}
    \frac{d}{dt}h(x(t))+\gamma(h(x(t)))\geq 0.
\end{equation}
Additionally, if Eq. \eqref{eqn:CBFcond} is satisfied for all $t\geq 0$, then the set $\S$ is forward invariant and asymptotically stable \cite{Ames17}.
Now by \eqref{eqn:controlsys}, it follows that $\dot{h}$ has the following form:
\begin{equation}
\label{eqn:cbfcondcontrol}
    \dot{h}(x(t),u(t)) =\frac{d}{dt}h(x(t))=\frac{\partial h}{\partial x}(x(t))f(x(t),u(t)).
\end{equation}
Therefore, the closed-loop system is safe if for every $t\geq 0$,
\begin{equation}
    \frac{\partial h}{\partial x}(x(t))f(x(t),u(t))+\gamma(h(x(t)))\geq 0.
\end{equation}
If this inequality is not satisfied for a particular $t\geq 0$, then the control law can be modified as follows to guarantee the safety of $\S$. Given $x\in \R^n$, define the set
$K_{x}\subset \R^m$ by:
\begin{align}
K_{x}= & \left\{u\in \R^m :  \frac{\partial h}{\partial x}(x)f(x,u)+\gamma(h(x))\geq 0\large \right\}.
\end{align}
Therefore, $u(t)$ given by: 
\begin{equation}
\label{eqn:optimizationbasedcontrol}
    u(t)\in \underset{u \in \R^{m}}{\operatorname{argmin}} \{||u-k(x(t))||^2~:~u\in K_{x(t)}\}
\end{equation}
ensures the safety of the closed-loop  system while modifying the control law defined in \eqref{eqn:feedbackcont} in a minimally invasive fashion.

Finally, note that if the dynamics of the plant are control-affine, namely
$f(x,u)=f_{0}(x)+f_{1}(x)u$ for  functions $f_{0}:\R^n\rightarrow \R^n$
and $f_{1}:\R^n\rightarrow \R^{n\times m}$, then $u(t)$, defined by \eqref{eqn:optimizationbasedcontrol},  can be computed by a quadratic program (QP).  In particular, we obtain an example of the feedback controller $u = k_{\mathrm{s}}(x)$ that ``filters'' the controller in \eqref{eqn:feedbackcont} and renders the system safe: 
\begin{align}
\label{eqn:CBFQP}
k_{\mathrm{s}}(x) = \underset{u \in \R^{m}}{\operatorname{argmin}} & ~ ||u-k(x)||^2 \\
\mathrm{s.t.} &  ~ \frac{\partial h}{\partial x}(x) f_0(x) + 
\frac{\partial h}{\partial x}(x) f_1(x) u \geq - \gamma(h(x)), \nonumber
\end{align}
where we suppressed the dependence on $t$.
This QP admits a real-time implementation (see \cite{Ames14}).
Importantly, this paper will allow for quadratic program representations of safety-critical CBF-based controllers even when the plant is not control affine.

\subsection{Tracking Control by a Newton-Raphson Flow}
In this paper, we consider a tracking-control technique based on a flow version of the Newton-Raphson method for solving algebraic equations \cite{Wardi19}. This provides us with a convenient expression of a feedback control law defined by an ordinary differential equation, as it is recalled in this section (a more-detailed discussions and analyses can be found in \cite{Wardi19}). It is worth noting, however, that the main results presented in this paper do not depend on this specific method employed.  We consider a specific method to define an integral control law that achieves tracking for definiteness.

To explain the main idea behind the tracking controller utilized, consider first the simple case where the plant is a memoryless nonlinearity, input-output system of the form $y(t)=g(u(t))$,
where $u(t)\in \R^m$ is the input, $y(t)\in \R^m$ is the output, and $g:\R^m\rightarrow \R^m$ is a continuously-differentiable function.
Given a continuously-differentiable reference signal $r(t)\in \R^m$, 
suppose that the objective is to design a control law such that $y(t)$ converges to $r(t)$ in a suitable sense as described below.  Consider the dynamically defined controller:
\begin{equation}
\label{eqn:udotnomem}
    \dot{u}(t)=\Big(\frac{\partial g}{\partial u}(u(t))\Big)^{-1}\big(r(t)-y(t)\big).
\end{equation}
This controller essentially implements a Newton-Raphson (NR) flow whose vector field at time $t$, $\dot{u}(t)$, is the direction defined by the classical NR method for
solving the  time-dependent algebraic equation: $r(t)-g(u)=0$. 

Next, suppose that the plant is dynamic, i.e., defined by the state equation \eqref{eqn:controlsys}, and consider the output equation $y(t)=\zeta(x(t))$,
where the function $\zeta:\R^n\rightarrow \R^m$ is continuously differentiable. Now $x(t)$, and hence $y(t)$, are functions of the initial condition $x_0$ and past control actions: $u(\tau),~\tau\in[0,t)$.  Yet, instantaneously, $y(t)$ is not a function of $u(t)$.
Therefore the controller $u(t)$ cannot be defined by an equation like \eqref{eqn:udotnomem}.  
However, given $t\geq 0$ and $T> 0$, $x(t+T)$ and hence $y(t+T)$ depend on $u(\tau),~\tau\in[t,t+T]$.  This observation forms the basis for our tracking controller. 

Given a time $t \geq 0$ and $T > 0$, fixing $u_t := u(t)$ as a constant over $[t,t + T]$ and (approximately) forward integrating \eqref{eqn:controlsys}, i.e., $\dot{\hat{x}}(\tau) = f(\hat{x}(\tau),u_t)$, over this interval with initial condition $\hat{x}(t)= x(t)$ results in prediction of the state $\hat{x}(t + T)$ and hence a prediction of the output: 
\begin{equation}
\label{eqn:predictor}
    \hat{y}(t+T)= \zeta(\hat{x}(t + T)) =: g(x(t),u(t))
\end{equation}
that, therefore, depends on $x(t)$ and $u(t)$.   
In this case, a suitable extension of \eqref{eqn:udotnomem} has the following form:
\begin{equation}
\label{eqn:controleqn}
    \dot{u}(t)=\alpha\Big(\frac{\partial g}{\partial u}(x(t),u(t)\Big)^{-1}\big(r(t)-\hat{y}(t+T)\big);
\end{equation}
where $\alpha > 0$ is a controller gain. 
In \cite{Wardi19}, the following convergence result was established: 
\begin{equation}
\underset{t\rightarrow\infty}{\operatorname{limsup}}
||r(t)-y(t)||<\eta_{1}+\eta_{2}/\alpha,
\end{equation}
where $\eta_{1}:={\rm limsup}||y(t)-\hat{y}(t)||$, and $\eta_{2}:={\rm limsup}||\dot{r}(t)||$, where limsup means as $t\rightarrow\infty$. Thus, increasing the controller gain, $\alpha$, can reduce the error due variations in $r(t)$, but cannot attenuate the effects of prediction errors.

\section{Integral Control Barrier Functions}
\label{sec:mainresult}
  
This section presents the main construction and results of the paper.  Specifically, we introduce integral control barrier functions, and demonstrate that they can enforce safety for dynamically defined control laws.  This, in essence, creates a new paradigm for safety-critical integral control.  

Consider a dynamical system defined by Eq. \eqref{eqn:controlsys} with an initial condition $x_{0}:=x(0)\in \R^n$. In a departure from the type of controller considered in Section II.A, which is algebraic, we define a general feedback law by the ordinary differential equation:
\begin{equation}
\label{eqn:dyncontrolequation}
  \dot{u}(t)=\phi(x(t),u(t),t),
\end{equation}
with an initial condition $u_{0}:=u(0)\in \R^m$.
We assume that the  function
$\phi:\R^n\times \R^m\times \R\rightarrow \R^m$ is  continuously differentiable. The closed-loop system is defined by
the state equation \eqref{eqn:controlsys} and the control equation \eqref{eqn:dyncontrolequation}.
We  write these two equations jointly in the following way: 
\begin{equation}
\label{eqn:stateeqn}
  \left[\begin{array}{l}
  \dot{x}(t)\\
  \dot{u}(t)
  \end{array}
  \right]~=~
  \left[
  \begin{array}{l}
  f(x(t),u(t))\\
  \phi(x(t),u(t),t)
  \end{array}
  \right].
\end{equation}
Define $z(t):=(x(t)^{\top},u(t)^{\top})^{\top}\in \R^n\times \R^m$ with $z(t)$ the {\it augmented state}, which can be viewed as the state of the closed-loop system.  Note that the state equation \eqref{eqn:stateeqn} has no external input.
  
 Let $\S\subset \R^n\times \R^m$ be a set, the {\it safety set}, defined as the 0-superlevel set of a continuously differentiable function $h:\R^{n}\times \R^m\rightarrow \R$ with $0$ a regular value. In addition, suppose the existence of a continuous extended class-${\cal K}$ function $\gamma:\R\rightarrow \R$ such that the following is satisfied along every trajectory $z(t)$ of the closed-loop system:
 \begin{equation}
 \label{eqn:zCBFcond}
     \frac{d}{dt}h(z(t))+\gamma(h(z(t)))\geq 0.
 \end{equation}
Then, the set $\S$ is forward invariant and asymptotically stable. 
 Throughout the remainder of this paper, for the sake of simplicity of exposition, we will suppress the dependence on time---this can easily be inferred from context. 
 
 For notational simplicity, define the following vector valued $p(x,u) \in \R^m$ and scalar valued $d(x,u) \in \R$ functions:  
  For given $t\geq 0$, define 
  the scalar $d $ and vector $p \in \R^m$ by:
 \begin{align}
 \label{eqn:pxu}
      p(x,u)  :=  &  \Big(\frac{\partial h}{\partial u}(x ,u )\Big)^{\top} \\
      \label{eqn:pdxut}
     d(x,u,t)   := &  -\Big(\frac{\partial h}{\partial x}(x ,u )f(x ,u )  +\frac{\partial h}{\partial u}(x ,u )\phi(x ,u ,t) \nonumber\\
     & \qquad \qquad  +\gamma(h(x ,u ))\Big),
 \end{align}
Note that Eq. \eqref{eqn:zCBFcond} can thus be recast as $d(x,u,t) \leq 0$. 
  
\newsec{Main Result.}
Traditional CBF methods cannot be directly applied to systems of the form given in \eqref{eqn:stateeqn} due to the dynamically defined control law, i.e., the control law $u$ is a result of integrating the augmented state equation in \eqref{eqn:stateeqn}.  
Our approach is to add an auxiliary input to the state equation so that, if the inequality in \eqref{eqn:zCBFcond} is not satisfied for a particular $t\geq 0$, we are able to determine the minimal modification of the dynamically defined control law that will guarantee safety. To this end, we modify the augmented-state equation
\eqref{eqn:stateeqn} to include the input $v\in \R^m$ as follows:
\begin{equation}
\label{eqn:stateeqnv}
  \left[\begin{array}{l}
  \dot{x}\\
  \dot{u}
  \end{array}
  \right]~=~
  \left[
  \begin{array}{l}
  f(x,u)\\
  \phi(x ,u ,t)+v 
  \end{array}
  \right].
\end{equation}

For every $z:=(x^{\top},u^{\top})^{\top}\in \R^n\times \R^m$
and $t\geq 0$, define:
\begin{align}
\label{eqn:Szset}
  K_{z,t}:= & \big\{v\in \R^m~
  :~\frac{\partial h}{\partial x}(x,u)f(x,u)    \\
   &+\frac{\partial h}{\partial u}(x,u)\big(\phi(x,u,t)+v\big) ~ +\gamma(h(x,u))~\geq~0\big\} \nonumber\\
  \label{eqn:Szsetshort} 
   =& \left\{ v\in \R^m~
  :~ p(x,u)^{\top} v \geq d(x,u,t) \right\}.
\end{align}
Therefore, if $v(x,t) \in K_{z,t}$, it follows that Eq. \eqref{eqn:zCBFcond} is satisfied.  Additionally, if the dynamic control law $\dot{u} = \phi(x,u,t)$ is inherently safe, then $d(x,u,t) \leq 0$ and $v = 0$ imply that the system is safe.
This leads to the formulation of I-CBFs,  
mirroring the classic definition of control Lyapunov functions \cite{sontag1989universal}) and similar in spirit to control-dependent CBFs \cite{huang2019guaranteed}.

\begin{definition}
For the system \eqref{eqn:stateeqnv}, with corresponding safe set $\S \subset \R^n \times \R^m$ defined as the 0-superlevel set of a function $h : \R^n \times \R^m \to \R$ with $0$ a regular value: $\S = \{(x,u) \in \R^n \times \R^m ~ : ~ h(x,u) \geq 0\}$.  Then $h$ is an \emph{integral control barrier function (I-CBF)} if for any $(x,u) \in \R^n \times \R^m$ and $t \geq 0$:
\begin{eqnarray}
\label{eqn:relativedegree}
p(x,u) = 0 \qquad \Rightarrow \qquad d(x,u,t) \leq 0.
\end{eqnarray}
\end{definition}
  
We now have the necessary constructions to state the main result of this paper. 
 
\begin{theorem}
\label{thm:main} 
 Consider the control system $\dot{x} = f(x,u)$, with $x \in \R^n$ and $u \in \R^m$, and suppose that there is a corresponding dynamically defined controller: $\dot{u} = \phi(x,u,t)$.  If the safe set $\S \subset \R^n \times \R^m$ is defined by an integral control barrier function, $h : \R^n \times \R^m \to \R$,
 then modifying the dynamically defined controller to be of the form: 
 \begin{eqnarray}
 \label{eqn:umodifiedopt}
 \dot{u}  =  \phi(x,u,t) + v^*(x,u,t)
 \end{eqnarray}
 with $v^*$ the solution to the QP: 
 \begin{align}
\label{eqn:CBFQPv}
 \qquad \qquad \quad  v^*(x,u,t)  =  \underset{v \in \R^{m}}{\operatorname{argmin}} & ~ ||v||^2 \\
\mathrm{s.t.} &  ~ p(x,u)^{\top} v \geq d(x,u,t) \nonumber
\end{align}
 results in safety, i.e., the control system $\dot{x} = f(x,u)$ with the dynamically defined controller \eqref{eqn:umodifiedopt} results in $\S$ being forward invariant: if $(x(0),u(0)) \in \S$ then $(x(t),u(t)) \in \S$ for all $t \geq 0$.
\end{theorem}

\begin{proof}
We only need to verify that Eq. \eqref{eqn:zCBFcond} is satisfied for $z(t) = (x(t),u(t))$ the solution to \eqref{eqn:stateeqnv} with $v = v^*(x,u,t)$.  Per \eqref{eqn:Szset} and \eqref{eqn:Szsetshort}, this is equivalent to: 
\begin{eqnarray}
\label{eqn:proofmaineqn}
p(x,u)^{\top} v^*(x,u,t) \geq d(x,u,t).
\end{eqnarray}
This, in turn, will imply that $(x(t),u(t)) \in \S$ for all $t \geq 0$ if $(x(0),u(0)) \in \S$ per the main result of \cite{Ames19}. 

Thus, to establish the result, we must show the solution to \eqref{eqn:CBFQPv} satisfies \eqref{eqn:proofmaineqn}.  This follows from the fact that we can obtain an explicit solution to \eqref{eqn:CBFQPv}.  In particular, the condition that $h$ is an integral control barrier function and, in particular, that it satisfies \eqref{eqn:relativedegree}, implies that the linear independent constraint qualification condition is satisfied \cite{boyd2004convex}.  Therefore, using the KKT optimality conditions (see \cite{xu2015robustness}), the solution to \eqref{eqn:CBFQPv} is given by the \emph{min-norm} controller: 
\begin{align}
\label{eqn:minnormcon}
v^*(x,u,t) = \left\{ 
\begin{array}{lcr}
\frac{d(x,u,t)}{\| p(x,u) \|^2} p(x,u) & \mathrm{if~} d(x,u,t) > 0 \\
0 & \mathrm{if~} d(x,u,t) \leq 0 
\end{array}
\right. 
\end{align}
This controller is well-defined because, as $h$ is a control barrier function and thus satisfies \eqref{eqn:relativedegree}, it follows that: $d(x,u,t) > 0$ implies that $p(x,u) \neq 0$.  Additionally, from this explicit form, one can verify the Lipschitz continuity of this controller (assuming Lipschitz continuity of $p$ and $d$).  

Finally, the explicit form of \eqref{eqn:minnormcon} makes it clear that it satisfies \eqref{eqn:proofmaineqn}.  If $d(x,u,t) \leq 0$ than it is naturally satisfied with $v^*(x,u,t) = 0$.  If $d(x,u,t) > 0$, then:
\begin{eqnarray}
\quad p(x,u)^{\top} v^*(x,u,t)  & = &   \underbrace{p(x,u)^{\top} p(x,u)}_{\| p(x,u) \|^2} \frac{d(x,u,t)}{\| p(x,u) \|^2} \nonumber\\
& = &  d(x,u,t). \hspace{2.6cm}  \hfill \qedhere \nonumber
\end{eqnarray}
\end{proof}

  \begin{remark}
  A natural consequence of this form of control barrier functions is, as its name suggests, its application to integral control.  In the case of ``pure'' integral control, we can consider a control system \eqref{eqn:stateeqnv} together with $\phi(x,u,t) \equiv 0$, i.e., the dynamic extension via the addition of an integrator.  In this case, the controller from Theorem \ref{thm:main} is just $u(t)~=~\int_{0}^t v^*(x(\tau),u(\tau),\tau) d \tau$.
  \end{remark}

\begin{remark}
The controller given in Theorem \ref{thm:main} can be viewed in the following fashion.  We began with a control system $\dot{x} = f(x,u)$ for which we synthesized a dynamic controller: 
\begin{align*}
\dot{u} = \phi(x,u,t) +  \left\{ 
\begin{array}{lcr}
\frac{d(x,u,t)}{\| p(x,u) \|^2} p(x,u) & \mathrm{if} & d(x,u,t) > 0 \\
0 & \mathrm{if} &  d(x,u,t) \leq 0 
\end{array}
\right. 
\end{align*}
with $\phi(x,u,t)$ the ``feedforward'' integral controller that is modulated by the additional term to ensure safety through a minimal modification.  The controller, $\phi$, for example, can be given by \eqref{eqn:controleqn}.  The advantages of the integral instantiation of control barrier functions are: they allow CBFs to be applied to systems not in control affine form; the dynamic equation describing $u$ is integrated thus smoothing out the non-smooth nature of solutions to QPs; they can encode input bounds.  
\end{remark}

\begin{remark}
\label{rmk:staticfeedback}
 We can also consider the case when we have a nominal controller $u = k(x)$ as in \eqref{eqn:feedbackcont}.  In this case, if we pick the following ``feedforward'' integral controller: 
 $$
 \phi(x,u,t) = \frac{\partial k}{\partial x}(x) f(x,u) + \frac{\alpha}{2}\left(k(x,t)-u\right),
 $$
 for $\alpha > 0$, we provably get that $\| u - k(x) \| \to 0$ exponentially.  This can, therefore, be coupled with Theorem \ref{thm:main} to achieve provable tracking of the desired controller subject to safety.  Proving this is not within the scope of this paper, but will be the subject of future work---as it indicates the ability to safely track desired controllers even for non-affine control systems. 
\end{remark}

\section{Applications and Extensions of I-CBFs}

In this section, we explore some extensions and applications of integral control barrier functions.  In particular, we demonstrate that I-CBFs can encode both input bounds and state constraints.  Importantly, we show that I-CBFs can simultaneously enforce both, thereby enabling input bounded provably safe integral control subject to feasibility. 

\subsection{Application to Input Constrained Systems}
Consider the case when we have a system $\dot{x} = f(x,u)$ subject to input constraints: for $u_{\max} > 0$, 
$$
\| u \|^2 \leq u_{\max} \qquad \Rightarrow \qquad h_u(u) := u_{\max} - u^{\top} u \geq 0. 
$$
Naturally, this is a conservative method for enforcing input constraints due to its scalar instantiation.  Yet, it allows us to view input bounds as safety relative to the set $\S_u := \{ (x,u) \in \R^n \times \R^m ~ : ~ h_u(u) \geq 0\}$.
In this case, \eqref{eqn:pxu} and \eqref{eqn:pdxut} become: 
 \begin{eqnarray}
 \label{eqn:pxu2}
      p_u(x,u)  &  =  &  - 2 u \nonumber\\
      \label{eqn:pdxut2}
     d_u(x,u,t)  &  =  &   2 u^{\top} \phi(x ,u ,t) - \gamma_u(h_u(u )). \nonumber
 \end{eqnarray}
Thus, following from Theorem \ref{thm:main}: 

\begin{lemma}
\label{lem:torquebound}
For the system \eqref{eqn:stateeqnv}, $h_u$ is an integral control barrier function, i.e., $\S_u$ can be rendered safe (forward invariant), with the integral control law \eqref{eqn:umodifiedopt} where: 
 \begin{align}
\label{eqn:CBFQPvubound}
  v^*(x,u,t)  =  \underset{v \in \R^{m}}{\operatorname{argmin}} & ~ ||v||^2 \\
\mathrm{s.t.} &  ~ 2 u^{\top} v \geq - 2 u^{\top} \phi(x ,u ,t) - \gamma_u(h_u(u ))
\nonumber
\end{align}
\end{lemma}

\begin{proof}
We only need to verify \eqref{eqn:relativedegree}.  If $p_u(x,u)= 2u = 0$, then $u = 0$, and therefore $d_u(x,u,t) = - \gamma_u(h_u(u))$.  But $u = 0$ implies that $(x,u) \in \S_u$, and therefore $h_u(u) \geq 0$ or $\gamma_u(h_u(u)) \geq 0$, and so $d_u(x,u,t) \leq 0$ as desired. 
\end{proof}

\subsection{Application to State Constrained Control Affine Systems}
\label{ssec:affine}

Returning to the systems that are often considered in the context of control barrier functions---affine control systems---we wish to understand the relationship between classical and integral control barrier functions for systems of this form.
Consider the variation of \eqref{eqn:stateeqnv} in control affine form:
\begin{eqnarray}
\label{eqn:stateeqnaffine}
\dot{x} & = & f_0(x) + f_1(x) (\mu +  u) \\
\dot{u} & = & \phi(x,u,t) \nonumber
\end{eqnarray}
where here we added an auxiliary control law in the $x$ dynamics---the reason for this will become apparent soon. 

For this system suppose we have a safe set defined in terms of state (and not input), i.e., by $h_x : \R^n \to \R$ with: 
$$
\Sx := \{ (x,u) \in \R^n \times \R^m ~ : ~ h_x(x) \geq 0\}.
$$
Additionally, suppose that $h_x$ is a valid control barrier functions for the affine dynamics without the corresponding integrator: $\dot{h}_x(x,u) \geq -\gamma_x(h_x(x))$.
This can be quantified by a conditions analogous to \eqref{eqn:relativedegree}, specifically:
\begin{align}
\label{eqn:hxCBFcond}
\underbrace{\frac{\partial h_x}{\partial x}(x) f_1(x)}_{:= p_x(x)^{\top}} = 0
\Rightarrow
\underbrace{
\frac{\partial h_x}{\partial x}(x) f_0(x)  +  \gamma_x(h_x(x))   }_{ := -d_x(x)}  \geq 0 
\end{align}
The following lemma gives a controller that will guarantee the forward invariance of $\Sx$ under the assumption that $h_x$ is a valid CBF. It will do so through a combined ``traditional'' CBF controller together with an integral controller. 

\begin{lemma}
\label{lem:statebound}
Given the system \eqref{eqn:stateeqnaffine} with $h_x : \R^n \to \R$ a control barrier function for $\dot{x} = f_0(x) + f_1(x) u$.  Then the controller: 
$$
k_{\mathrm{s}}(x,u,t) = \mu^*(x(t),u(t)) + \int_{0}^{t} \phi(x(\tau),u(\tau),\tau) d \tau
$$
where, for any nominal controller $u = k(x)$: 
\begin{align}
\label{eqn:CBFQPint}
\mu^*(x,u) = \underset{\mu \in \R^{m}}{\operatorname{argmin}}&   ~ ||\mu +u -k(x)||^2 \\
\mathrm{s.t.}  ~ &  ~
p_x(x)^{\top} \mu  \geq d_x(x) - p_x(x)^{\top} u \nonumber
\end{align}
renders the set $\Sx$ forward invariant, i.e., safe.  
\end{lemma}

\begin{proof}
By \eqref{eqn:hxCBFcond}, the QP \eqref{eqn:CBFQPint} is well-defined and has a solution.  This follows from the fact that $u$ not being a decision variable anymore does not affect \eqref{eqn:hxCBFcond}, i.e., $p_x(x) = 0$ still implies that $d_x(x) \leq 0$ since the term $-p_x(x) u$ vanishes.  Thus safety, i.e., forward invariance of $\Sx$, is guarenteed by the classic CBF result \cite{Ames17}. 
\end{proof}

\subsection{Extension to Multiple Control Barrier Functions}

With the goal of now simultaneously enforcing state and input constraints in a holistic fashion,
we will begin by taking inspiration from Section \ref{ssec:affine}, wherein we must achieve similar results for systems with only integral control. 
Returning to original state equations \eqref{eqn:stateeqnv}, but in affine form:
\begin{eqnarray}
\label{eqn:stateeqnaffinetwop}
\dot{x} & = & f_0(x) + f_1(x) u \\
\dot{u} & = & \phi(x,u,t) + v,  \nonumber
\end{eqnarray}
consider state constraints encoded by $\Sx$, i.e., by $h_x(x) \geq 0$. The input $v$ no longer appears in $\dot{h}_x$, but, the forward invariance of the set $\Sx$ is quantified by the condition $\dot{h}_x(x,u)  + \gamma_x(h_x(x)) \geq 0 $. Thus, in order to ensure the latter inequality is satisfied, we follow the approach in \cite{nguyen2016exponential} and let:
 \begin{eqnarray}
 \label{eqn:heeqn}
\he(x,u)   :=  \dot{h}_x(x,u)  + \gamma_x(h_x(x)) =  p_x(x)^{\top} u - d_x(x). 
\end{eqnarray}
The result is a function: $\he : \R^n \times \R^m \to \R$ with corresponding set: $\Se = \{ (x,u) \in \R^n \times \R^m ~ : ~ \he(x,u) \geq 0\}$.   
Let $\pe(x,u)$ and $\de(x,u,t)$ be the corresponding functions defined from $\he$, for some $\gamma_e$, as in \eqref{eqn:pxu} and \eqref{eqn:pdxut}. 

\begin{lemma}
\label{lem:cbftoicbf}
Given the system \eqref{eqn:stateeqnaffinetwop} with $h_x : \R^n \to \R$ a control barrier function with $p_x(x) \neq 0$, i.e., $h_x$ has relative degree 1, then $h_e(x,u) = \dot{h}_x(x,u)  + \gamma_x(h(x))$ is an integral control barrier function.  Additionally, for the I-CBF control law \eqref{eqn:CBFQPv}, using $\pe$ and $\de$ in place of $p$ and $d$, if:
$$
\begin{array}{r}
x_0 \in \Sx  \\
(x_0,u_0) \in \Se
\end{array}
\quad \Rightarrow \quad
\begin{array}{rr}
x(t) \in \Sx  & \quad \forall~ t \geq 0 \\
(x(t),u(t)) \in \Se &  \quad \forall~ t \geq 0
\end{array}
$$
That is, $\Sx$ is safe subject to appropriate initial conditions.
\end{lemma}

\begin{proof}
It is easy to verify that $p_e(x,u) = p_x(x)$, thus the relative degree condition transfers from $h$ to $\he$.  As a result, the controller in \eqref{eqn:CBFQPv} is well-defined, i.e., \eqref{eqn:relativedegree} is trivially satisfied.  As a result, $\Se$ is forward invariant, i.e., assuming an initial condition $(x_0,u_0) \in \Se$ then $\dot{h}_e(x(t),u(t)) = \dot{h}_x(x(t),u(t)) + \gamma(h_x(x(t))) \geq 0$.  Coupling this with the assumption that $x_0 \in \Sx$, i.e, that $h_x(x_0) \geq 0$, it follows that $\Sx$ is also forward invariant.
\end{proof}

\newsec{Combining State Constraints with Input Bounds.} With the I-CBF $h_e$, obtained from $h_x$ via Lemma \ref{lem:cbftoicbf}, we can synthesize controllers that 
enforce state constraints with input bounds.

\begin{theorem}
\label{thm:maintwo}
Consider the system in \eqref{eqn:stateeqnaffinetwop}, together with corresponding CBF $h_x : \R^n \to \R$ and I-CBF $h_u : \R^m \to \R$, and safe sets $\Sx$ and $\Su$ respectively, and let $h_e : \R^n \times \R^m$ be given as in \eqref{eqn:heeqn}.   If the following QP is feasible: 
\begin{align}
\label{eqn:CBFQPintfinal}
v^*(x,u,t) 
= &\underset{v\in \R^{m}}{\operatorname{argmin}}   ~  ||v||^2 \\
& \quad \mathrm{s.t.}  \quad 
\begin{bmatrix}
p_x(x)^{\top}  \\
p_u(x,u)^{\top}
\end{bmatrix}
v \geq 
\begin{bmatrix}
d_e(x,u,t) \\
d_u(x,u,t) 
\end{bmatrix}  
\nonumber
\end{align}
then the dynamically defined (integral) controller
$$\dot{u}~=~\phi(x,u,t) + v^*(x,u,t)$$
renders the set $\S = \Sx \cap \Su$ forward invariant, i.e., safe, for appropriate initial conditions: 
$$
\begin{array}{r}
x_0 \in \Sx  \\
(x_0,u_0) \in \Se \cap \Su
\end{array}
\quad \Rightarrow \quad
\begin{array}{r}
(x(t),u(t)) \in \Sx \cap \Su \\
\forall ~ t \geq 0. 
\end{array}
$$
\end{theorem}

\begin{proof}
The constraints in \eqref{eqn:CBFQPintfinal} are simply reformulations of: 
\begin{eqnarray}
\hspace{-0.1cm}
\begin{array}{l}
 \dot{h}_e(x,u, v)  \geq  - \gamma_e(h_e(x,u))  \\  
\dot{h}_u(x,u, v)  \geq   - \gamma_u(h_u(u))
\end{array}
\Leftrightarrow 
\begin{array}{r}
 p_x(x)^{\top} v   \geq   d_e(x,u,t)  \\  
p_u(x,u)^{\top} v  \geq  d_u(x,u,t) 
\end{array}   \nonumber
\end{eqnarray}
Therefore, assuming a feasible QP implies that $\Se \cap \Su$ is forward invariant by Theorem \ref{thm:main} and Lemma \ref{lem:torquebound}.  The result follows by Lemma \ref{lem:cbftoicbf}. 
\end{proof}

\begin{remark}
The requirement on the feasibility of the QP in Theorem \ref{thm:maintwo} encodes the fact that we are not assuming that $\Sx$ is a control invariant set.  Without this assumption, we do not know if there are feasible control inputs rendering $\Sx$ invariant---this is encoded by the feasibility of the QP. Moreover, if the QP is feasible, conditions for the Lipschitz continuity of its solution can be found (see, e.g., \cite{Ames17}). This and the relaxation of the feasibility assumption will be the subject of future work.
\end{remark}

\section{Simulation Results}

To demonstrate the results of the paper, we will return to an early motivating example considered for CBFs \cite{Ames14}: \emph{adaptive cruise control (ACC).}
The dynamics are given by: 
\begin{eqnarray}
\label{eqn:accdynamics}
\dot{x} = 
\begin{bmatrix}
x_2 \\
- \frac{1}{m} F_r(x) \\
v_0 - x_2 
\end{bmatrix}
+ 
\begin{bmatrix}
0 \\
\frac{1}{m} \\
0
\end{bmatrix} u
\end{eqnarray}
where $(x_1,x_2)$ are the position and velocity ($x_2 = \dot{x}_1$) of the vehicle, $m$ its mass, $x_3$ is the distance between the vehicle and the lead vehicle traveling at a velocity of $v_0$, and $F_r(x) = c_0 + c_1 x_2 + c_2 x_2^2$ is the empirical form of rolling resistance. The control objective for the system is to drive the car to a desired speed ($x_2 \to v_d$).  This can be represented as an output: $y = \zeta(x) = x_2 - v_d$.  To obtain the predictor as in \eqref{eqn:predictor}, we can forward integrate \eqref{eqn:accdynamics} with $c_2 = 0$, i.e., the linear approximation, resulting in: 
\[
\hat{y}(t + T) =  
{\scriptstyle      
-c_{1}^{-1}\left(c_{0}-u(t)+ m v_d-c_{1} e^{-\frac{c_{1} T }{m}} \left(x_2(t) +\frac{c_{0}-u(t)+m v_d}{c_{1}}\right)\right),
}
\]
yielding the dynamically defined control law per \eqref{eqn:controleqn}:
\[
\dot{u}(t) = \alpha c_1 \left( e^{\frac{-c_1}{m} T} - 1 \right)^{-1} \hat{y}(t + T)  =: \phi(x,u,t). 
\]
The state safety constraints encode the ``half the speedometer rule'' which yields the CBF (since $p_x(x) = -1.8/m \neq 0$):
\[
h_x(x) = x_3 - 1.8 x_2 \geq 0 \qquad \mapsto \qquad 
\Sx = \{h_x(x) \geq 0\} .
\]
The input constraint is given by the constraint that the wheel force is bounded by $|u| \leq m c_{a/d} g $ yielding an I-CBF:
\[
h_u(x) =  (m c_{a/d} g)^2 -  u^2 \geq 0 
\quad \mapsto \quad 
\Su = \{h_u(u) \geq 0\} .
\]
\noindent where $c_{a/d}$ is the factor of $g$ for acceleration/deceleration.  In both cases, we pick $\gamma_x(r) = \gamma_u(r) = \gamma r$ for $\gamma > 0$.  From $h_x$ we get $h_e$ as in \eqref{eqn:heeqn}, and we pick $\gamma_e(r) = \frac{1}{2} \gamma r$. 
\twocolumn[%
{\centering\includegraphics[width=\textwidth]{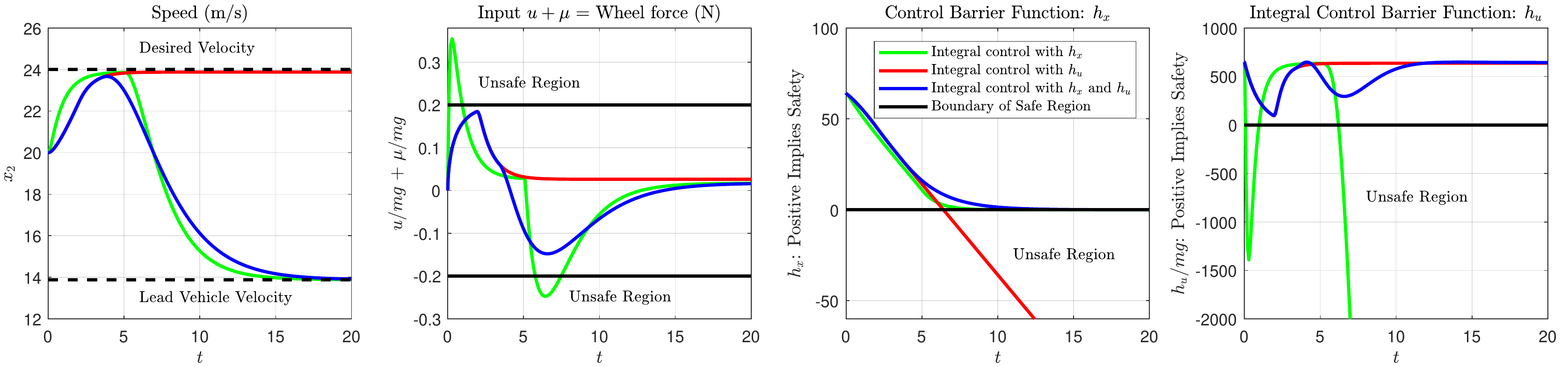}
\captionof{figure}{Methods developed in the paper applied to adaptive cruise control: the goal is for a vehicle to reach a desired velocity ($v_d = 24 m/s$) while not colliding with the lead vehicle traveling slower (at $14 m/s$) and not exceeding a maximum wheel force (input bounds).}
\label{fig:finalfig}}
\vspace{0.5cm}
]
Utilizing the constructions leading to Theorem \ref{thm:maintwo} results in the forward invariance of $\Sx \cap \Su$.  This is illustrated in Fig. \ref{fig:finalfig}, where the parameters and initial condition where chosen to match \cite{Ames14} with $\alpha = 10$ and $\gamma = 1$ above.
Three different controllers are plotted in Fig. \ref{fig:finalfig}, including: only input constraints and thus an I-CBF, $h_u$, with the controller from Lemma \ref{lem:torquebound} (red), only state constraints and thus a CBF, $h_x$, with the controller from Lemma \ref{lem:statebound} (green), and both input and state constraints, thus enforcing both $h_u$ and $h_e$, via the controller from Theorem \ref{thm:maintwo} (blue).  Therefore, via the presented unified controller, we are able to simultaneously satisfy the input and state constraints, and thus render the system safe for both constraints, i.e., render $\Sx \cap \Su$ forward invariant.  Importantly, when compared against past uses of this example in \cite{Ames19,Ames14} we can do so holistically.

\section{Conclusions}

This paper introduced integral control barrier functions (I-CBFs).  By considering dynamically defined controllers, we were able to guarantee safety using I-CBFs defined in terms of both state and input.  
This was applied in the context of dynamically defined tracking controllers for general nonlinear control systems (not necessarily control affine), wherein I-CBFs lead to minimal modification of these controllers via I-CBF based QPs (Theorem \ref{thm:main}).  
We then considered the specific cases of input bounds (Lemma \ref{lem:torquebound}) and state constraints (Lemma \ref{lem:cbftoicbf}), wherein both can be independently enforced via I-CBFs.  
Additionally, we proved the state constraints and input bounds can be simultaneously satisfied via the framework of I-CBFs assuming a feasible QP (Theorem \ref{thm:maintwo}).  This gives a holistic method for provably enforcing both through the use of safety-critical integral controllers.

Future work will be devoted to expanding this theoretic basis and demonstrating this theory experimentally.  From a theoretic perspective, understanding when the QP in Theorem \ref{thm:maintwo} is feasible is a rich problem that has important implications.  
Additionally, we wish to improve the means in which nominal controllers can be enforced in conjunction with integral control barrier functions.  
From an experimentally perspective, the ability of I-CBFs to bound inputs while achieving safety in state has important ramifications that we wish to explore on everything from walking robots to multi-robot systems to safey-critical autonomy.

\bibliographystyle{IEEEtran}
\bibliography{bib/Paper,bib/barrier}

\begin{thebibliography}{10}
\providecommand{\url}[1]{#1}
\csname url@samestyle\endcsname
\providecommand{\newblock}{\relax}
\providecommand{\bibinfo}[2]{#2}
\providecommand{\BIBentrySTDinterwordspacing}{\spaceskip=0pt\relax}
\providecommand{\BIBentryALTinterwordstretchfactor}{4}
\providecommand{\BIBentryALTinterwordspacing}{\spaceskip=\fontdimen2\font plus
\BIBentryALTinterwordstretchfactor\fontdimen3\font minus
  \fontdimen4\font\relax}
\providecommand{\BIBforeignlanguage}[2]{{%
\expandafter\ifx\csname l@#1\endcsname\relax
\typeout{** WARNING: IEEEtran.bst: No hyphenation pattern has been}%
\typeout{** loaded for the language `#1'. Using the pattern for}%
\typeout{** the default language instead.}%
\else
\language=\csname l@#1\endcsname
\fi
#2}}
\providecommand{\BIBdecl}{\relax}
\BIBdecl

\bibitem{Isidori90}
A.~Isidori and C.~Byrnes, ``Output regulation of nonlinear systems,''
  \emph{IEEE Transactions on Automatic Control}, vol.~35, pp. 131--140, 1990.

\bibitem{freeman1995robust}
R.~A. Freeman and P.~V. Kokotovic, ``Robust integral control for a class of
  uncertain nonlinear systems,'' in \emph{34th IEEE Conference on Decision and
  Control}, vol.~3, 1995, pp. 2245--2250.

\bibitem{khalil2000universal}
H.~K. Khalil, ``Universal integral controllers for minimum-phase nonlinear
  systems,'' \emph{IEEE Transactions on automatic control}, vol.~45, no.~3, pp.
  490--494, 2000.

\bibitem{jiang2001robust}
Z.-P. Jiang and I.~Marcels, ``Robust nonlinear integral control,'' \emph{IEEE
  Transactions on Automatic Control}, vol. 46(8), pp. 1336--1342, 2001.

\bibitem{Wardi19}
Y.~Wardi, C.~Seatzu, J.~Cortes, M.~Egerestedt, S.~Shivam, and I.~Buckley,
  ``Tracking control by the newton-raphson method with output prediction and
  controller speedup,'' in \emph{arxiv, {\rm
  http://arxiv.org/abs/1910.00693},}, 2019.

\bibitem{Ames19}
A.~Ames, S.~Coogan, M.~Egerstedt, G.~Notomista, K.~Sreenath, and P.~Tabuada,
  ``Control barrier functions: Theory and applications,'' in \emph{European
  Control Conference, {\rm Napoli, Italy, June 25-28}}, 2019.

\bibitem{huang2019guaranteed}
Y.~Huang, S.~Z. Yong, and Y.~Chen, ``Guaranteed vehicle safety control using
  control-dependent barrier functions,'' in \emph{2019 American Control
  Conference (ACC)}.\hskip 1em plus 0.5em minus 0.4em\relax IEEE, 2019, pp.
  983--988.

\bibitem{Ames14}
A.~Ames, J.~Grizzle, and P.~Tabuada, ``Control barrier function based quadratic
  programs with application to adaptive cruise control,'' in \emph{IEEE Conf.
  Decision and Control (CDC),}, 2014, p. 6271–6278.

\bibitem{Ames17}
A.~Ames, X.~Xu, J.~Grizzle, and P.~Tabuada, ``Control barrier function based
  quadratic programs for safety critical systems,'' \emph{IEEE Trans. Automatic
  Control}, vol.~62, p. 3861–3876, 2017.

\bibitem{jankovic2018robust}
M.~Jankovic, ``Robust control barrier functions for constrained stabilization
  of nonlinear systems,'' \emph{Automatica}, vol.~96, pp. 359--367, 2018.

\bibitem{xu2015robustness}
X.~Xu, P.~Tabuada, J.~W. Grizzle, and A.~D. Ames, ``Robustness of control
  barrier functions for safety critical control,'' \emph{IFAC-PapersOnLine},
  vol.~48, no.~27, pp. 54--61, 2015.

\bibitem{nguyen2016exponential}
Q.~Nguyen and K.~Sreenath, ``Exponential control barrier functions for
  enforcing high relative-degree safety-critical constraints,'' in \emph{2016
  American Control Conference (ACC)}.\hskip 1em plus 0.5em minus 0.4em\relax
  IEEE, 2016, pp. 322--328.

\bibitem{wang2016multi}
L.~Wang, A.~D. Ames, and M.~Egerstedt, ``Multi-objective compositions for
  collision-free connectivity maintenance in teams of mobile robots,'' in
  \emph{Conference on Decision and Control}, 2016, pp. 2659--2664.

\bibitem{Wang17}
L.~Wang, A.~Ames, and M.~Egerstedt, ``Safety barrier certificates for
  collisions-free multi-robot systems,'' \emph{IEEE Transactions on Robotics},
  vol.~33, no.~3, pp. 661--674, 2017.

\bibitem{lindemann2018control}
L.~Lindemann and D.~V. Dimarogonas, ``Control barrier functions for signal
  temporal logic tasks,'' \emph{IEEE control systems letters}, vol.~3, no.~1,
  pp. 96--101, 2018.

\bibitem{gilbert2002nonlinear}
E.~Gilbert and I.~Kolmanovsky, ``Nonlinear tracking control in the presence of
  state and control constraints: a generalized reference governor,''
  \emph{Automatica}, vol.~38, no.~12, pp. 2063--2073, 2002.

\bibitem{nicotra2018explicit}
M.~M. Nicotra and E.~Garone, ``The explicit reference governor: A general
  framework for the closed-form control of constrained nonlinear systems,''
  \emph{IEEE Control Systems Magazine}, vol.~38, no.~4, pp. 89--107, 2018.

\bibitem{tee2012control}
K.~P. Tee and S.~S. Ge, ``Control of state-constrained nonlinear systems using
  integral barrier lyapunov functionals,'' in \emph{51st IEEE Conference on
  Decision and Control (CDC)}, 2012, pp. 3239--3244.

\bibitem{sontag1989universal}
E.~D. Sontag, ``A ‘universal’construction of {A}rtstein's theorem on
  nonlinear stabilization,'' \emph{Systems \& control letters}, vol.~13, no.~2,
  pp. 117--123, 1989.

\bibitem{boyd2004convex}
S.~Boyd and L.~Vandenberghe, \emph{Convex optimization}.\hskip 1em plus 0.5em
  minus 0.4em\relax Cambridge university press, 2004.

\end{thebibliography}

\end{document}